\documentclass[a4paper,12pt]{article}
\usepackage{latexsym,amsmath,amssymb,amsthm}
\theoremstyle{plain}
\newtheorem{theorem}{Theorem}[section]
\newtheorem{lemma}[theorem]{Lemma}

\newtheorem{definition}[theorem]{Definition}

\usepackage{color}
\date{}

\title{A note on some best proximity point theorems proved under P-property}
\author{Ali Abkar\footnote{Department of Mathematics, Imam Khomeini International University,
34149 Qazvin, Iran; email: abkar@ikiu.ac.ir}, Moosa Gabeleh\footnote{Corresponding author:
Department of Mathematics, Ayatollah Boroujerdi
University, Boroujerd, Iran; email:
gab.moo@gmail.com}}
\begin{document}
\maketitle

 \noindent{\bf Abstract.}
 In this article, we show that some recent results on the existence of best proximity points
 can be obtained from the same results in fixed point theory.


\section{Introduction}
Let $A$ and $B$ be two nonempty subsets of a metric space $(X,d)$. In this paper,
we adopt the following notations and definitions.
$$D(x,B):=\inf\{d(x,y) : y\in B\}, \ \ \text{for all} \ \ x\in X,$$
$$A_0:=\{x\in A : d(x,y)=dist(A,B), \ \ \text{for some} \ \ y\in B\},$$
$$B_0:=\{y\in B : d(x,y)=dist(A,B), \ \ \text{for some} \ \ x\in A\}.$$

The notion of \emph{best proximity point} is defined as follows.

\begin{definition}
 Let $A$ and $B$ be nonempty subsets of a metric space $(X,d)$ and $T:A\rightarrow B$
 be a non-self mapping. A point $x^*\in A$ is called a best proximity point of $T$ if
  $d(x^*,Tx^*)=dist(A,B)$,
 where
 $$
 dist(A,B):=\inf\{d(x,y) : (x,y)\in A\times B\}.
 $$
\end{definition}

Similarly, for a multivalued non-self mapping $T:A\to 2^B$, where
$(A,B)$ is a nonempty pair of subsets of a metric space $(X,d)$, a
point $x^*\in A$ is a best proximity point of $T$ provided that
$D(x^*,Tx^*)=dist(A,B)$.

Recently, the notion of P-property was introduced in \cite{Raj} as follows.

\begin{definition}(\cite{Raj})
Let $(A,B)$ be a pair of nonempty subsets of a metric space
$(X,d)$ with $A_0\neq\emptyset$. The pair $(A,B)$ is said to have
P-property if and only if
\begin{align*}\begin{cases}d(x_1,y_1)=dist(A,B)\\
d(x_2,y_2)=dist(A,B)\end{cases}\Longrightarrow
d(x_1,x_2)=d(y_1,y_2),\end{align*} where $x_1,x_2\in A_0$ and
$y_1,y_2\in B_0$.
\end{definition}

By using this notion, some best proximity point results were proved for various classes of non-self mappings.
Here, we state some of them.

\begin{theorem}(\cite{Raj})
Let $(A,B)$ be a pair of nonempty closed subsets of a complete metric space $X$
such that $A_0$ is nonempty. Let $T:A\rightarrow B$ be a weakly contractive non-self mapping,
that is,
$$
d(Tx,Ty)\leq d(x,y)-\phi(d(x,y)) \ \forall x,y\in A,
$$
where $\phi:[0,\infty)\rightarrow[0,\infty)$ is a continuous and nondecreasing function
such that $\phi$ is positive on $(0,\infty), \phi(0)=0$ and $\lim_{t\rightarrow\infty}\phi(t)=\infty$.
Assume that the pair $(A,B)$ has the P-property and $T(A_0)\subseteq B_0$. Then
$T$ has a unique best proximity point.
\end{theorem}

\begin{theorem}(\cite{AG6})
Let $(A,B)$ be a pair of nonempty closed subsets of a Banach space $X$
such that $A$ is compact and $A_0$ is nonempty. Let $T:A\rightarrow B$ be a nonexpansive mapping,
that is
$$
\|Tx-Ty\|\leq\|x-y\| \ \forall x,y\in A.
$$
Assume that the pair $(A,B)$ has the P-property and $T(A_0)\subseteq B_0$. Then
$T$ has a best proximity point.
\end{theorem}

\begin{theorem}(\cite{B})
Let $(A,B)$ be a pair of nonempty closed subsets of a complete metric space $X$
such that $A_0$ is nonempty. Let $T:A\rightarrow B$ be a Meir-Keeler non-self mapping,
that is, for all $x,y\in A$ and for any $\varepsilon>0$, there exists $\delta(\varepsilon)>0$ such that
$$
d(x,y)<\varepsilon+\delta \quad{implies} \quad d(Tx,Ty)\leq\varepsilon.
$$
Assume that the pair $(A,B)$ has the P-property and $T(A_0)\subseteq B_0$. Then
$T$ has a unique best proximity point.
\end{theorem}

\begin{theorem}(\cite{AG8})
Let $(A,B)$ be a pair of nonempty closed subsets of a complete
metric space $(X,d)$ such that $A_0\neq\emptyset$ and $(A,B)$
satisfies the P-property. Let $T:A\rightarrow 2^B$ be a
multivalued contraction non-self mapping, that is,
\begin{equation*}
H(Tx,Ty)\leq\alpha d(x,y),
\end{equation*}
for some $\alpha\in(0,1)$ and for all $x,y\in A$. If $Tx$ is
bounded and closed in $B$ for all $x\in A$, and $Tx_0$ is included
in $B_0$ for each $x_0\in A_0$, then $T$ has a best proximity
point in $A$.
\end{theorem}

\section{Main Result}

In this section, we show that the existence of a best proximity
point in the main theorems of \cite{AG6, AG8, B, Raj}, can be
obtained from the existence of the fixed point for a self-map. We
begin our argument with the following lemmas.

\begin{lemma}(\cite{G})
Let $(A,B)$ be a pair of nonempty closed subsets of a complete
metric space $(X,d)$ such that $A_0$ is nonempty and $(A,B)$ has the P-property.
Then $(A_0,B_0)$ is a closed pair of subsets of $X$.
\end{lemma}

\begin{lemma}
Let $(A,B)$ be a pair of nonempty closed subsets of a metric space $(X,d)$
such that $A_0$ is nonempty.
Assume that the pair $(A,B)$ has the P-property. Then
there exists a bijective isometry $g:A_0\to B_0$ such that $d(x,gx)=dist(A,B)$.
\end{lemma}

\begin{proof}
Let $x\in A_0$, then there exists an element $y\in B_0$ such that
$$d(x,y)=dist(A,B).$$ Assume that there exists another point $\acute{y}\in B_0$
such that
$$d(x,\acute{y})=dist(A,B).$$ By the fact that $(A,B)$ has the P-property, we conclude that
$y=\acute{y}$. Consider the non-self mapping $g:A_0\to B_0$ such
that $d(x,gx)=dist(A,B)$. Clearly, $g$ is well defined. Moreover,
$g$ is an isometry. Indeed, if $x_1,x_2\in A_0$ then
$$
d(x_1,gx_1)=dist(A,B) \quad \& \quad d(x_2,gx_2)=dist(A,B).
$$
Again, since $(A,B)$ has the P-property,
$$d(x_1,x_2)=d(gx_1,gx_2),$$
 that is, $g$ is an isometry.
\end{proof}

Here, we prove that the existence and uniqueness of the best proximity point in
Theorem 1.3 is a sample result
of the existence of fixed point for a weakly contractive self-mapping.

\begin{theorem}
Let $(A,B)$ be a pair of nonempty closed subsets of a complete metric space $X$
such that $A_0$ is nonempty. Let $T:A\rightarrow B$ be a weakly contractive mapping.
Assume that the pair $(A,B)$ has the P-property and $T(A_0)\subseteq B_0$. Then
$T$ has a unique best proximity point.
\end{theorem}

\begin{proof}
Consider the bijective isometry $g:A_0\to B_0$ as in Lemma 2.2.
Since $T(A_0)\subseteq B_0$, for the self-mapping $g^{-1}T:A_0\to
A_0$ we have
$$
d(g^{-1}(Tx),g^{-1}(Ty))=d(Tx,Ty)\leq\varphi(d(x,y)),
$$
for all $x,y\in A_0$ which implies that the self-mapping $g^{-1}T$
is weakly contractive. Note that $A_0$ is closed by Lemma 2.1.
Thus, $g^{-1}T$ has a unique fixed point (\cite{R}). Suppose that
$x^*\in A_0$ is a unique fixed point of the self-mapping
$g^{-1}T$, that is, $g^{-1}T(x^*)=x^*$. So, $Tx^*=gx^*$ and then
$$
d(x^*,Tx^*)=d(x^*,gx^*)=dist(A,B),
$$
from which it follows that $x^*\in A_0$ is a unique best proximity
point of the non-self weakly contractive mapping $T$.
\end{proof}

\noindent{\bf Remark 2.1.} By a similar argument, using the fact
that every nonexpansive self-mapping defined on a nonempty compact
and convex subset of a Banach space has a fixed point, we conclude
Theorem 1.4. Also, existence and uniqueness of the best proximity
point for Meir-Keeler non-self mapping $T$ follows from the
Meir-Keeler's fixed point theorem (\cite{MK}). Finally, in Theorem
1.5, the Nadler's fixed point theorem (\cite{N}), ensures the
existence of a best proximity point for multivalued non-self $T$.

\end{document}